\documentclass[12pt]{article}

\newcommand{\ti}{\;\,\makebox[0pt]{$\mid$}\makebox[0pt]{$\cap$}\;\,} 

\usepackage{amsmath,  amsthm, amssymb}  

\begin{document}

\newtheorem{thm}{Theorem}
\newtheorem{cor}{Corollary}
\newtheorem{prop}{Proposition}
\newtheorem{lem}{Lemma}
\newtheorem{rem}{Remark}
\newtheorem{claim}{Claim}

\title{An example of  $C^1$-generically wild homoclinic classes with index deficiency}

\author{Katsutoshi Shinohara}

\maketitle

\begin{abstract}
Given a closed smooth four-dimensional manifold,
we construct a diffeomorphism that has a homoclinic class
whose continuation locally generically satisfies the following condition:
it does not admit any kind of dominated splittings
whereas any periodic points belonging to 
it never have unstable index one.
\end{abstract}

\section{Introduction}

\subsection{Preliminaries}
\subsubsection{General notations}
We consider a  closed (compact and boundaryless) smooth manifold $M$ with a Riemannian metric. 
For $A \subset M$, we denote its topological interior as $\mathrm{int}(A)$ and 
topological closure as $\overline{A}$.
We denote the group of $C^1$-diffeomorphisms of $M$ as $\mathrm{Diff}^1(M)$.
We fix a distance function on $\mathrm{Diff}^1(M)$ derives from the Riemannian metric and
furnish $\mathrm{Diff}^1(M)$ with the $C^1$-topology defined by the distance function.

By $TM$ we denote the tangent bundle of $M$.
By $\Lambda^k(TM)$ we denote the exterior product of $TM$ with degree $k$.
We furnish this bundle with the metric canonically induced from the Riemannian
metric on $M$.
For $f \in \mathrm{Diff}^1(M)$, by $\Lambda^k(df)$ we denote the bundle map
of $\Lambda^k(TM)$ canonically induced from $df$.
The support of $f$ is defined to be the closure of the set $\{ x \in M \mid f(x) \neq x\}$.
Given two submanifold $N_1, N_2 \subset M$, 
by $N_1 \ti N_2$ we denote the set of points in $N_1 \cap N_2$
at which $N_1$ and $N_2$ intersect transversally.

For $f : V \to W$, where $V$, $W$ are some finite dimensional Euclidean space and $f$ is a linear map, we define the value 
$m(f)$ to be the minimum of $\| f(v) \|$, where $v$ ranges all the unit vector in $V$.
Note that $f$ is injective if and only if $m(f) > 0$. When $W=V$ and $m(f) > 1$ then 
every absolute value of the eigenvalue of $f$
is greater than one.

\subsubsection{General notations related to dynamical systems}
Let $P$ be a hyperbolic periodic point of $f$.
The index of a hyperbolic periodic point $P$ (denoted as $\mathrm{ind}(P)$)
is defined to be the dimension of the unstable manifold of $P$.
By $W^s(P, f)$ (resp. $W^u(P, f)$) we denote the stable (resp. unstable) manifold of $P$. 
We often use the simplified notation $W^s(P)$ (resp. $W^u(P)$). 
By $H(P, f)$, or simply $H(P)$, we denote the homoclinic class of $P$, i.e., the closure of the set of 
points of transversal intersections of $W^u(P)$ and $W^s(P)$,
more precisely, $H(P) := \overline{W^s(P) \ti W^u(P)}$.
For a homoclinic class $H(P)$, its index set, denoted as $\mathrm{ind}(H(P))$,
is the collection of integers that appears as the 
index of some periodic point in $H(P)$.

If $g$ is a $C^1$-diffeomorphism sufficiently close to $f$,
then one can define the continuation of $P$.
We denote the continuation of $P$ for $g$ as $P(g)$.
We sometimes use the notation like $W^s(P, g)$ in the sense of $W^s(P(g), g)$.

\subsubsection{Dominated splittings}
Let $\Sigma$ be an $f$-invariant subset of $M$.
We say that $\Sigma$ admit $l$-dominated splitting if there exist two $df$-invariant 
subbundle $F$, $G$ of $TM|_{\Sigma}$ such that $TM|_{\Sigma} = F \oplus G$
and for all $x \in \Sigma$ following holds:
\[
\| df^l(x)|_{F} \|  \|(df)^{-l}(f^{l}(x))|_{G}\| < 1/2,
\]
where $\| \, \cdot \, \|$ denotes the operator norm derives from the
Riemannian metric.
We say $\Sigma$ does not admit dominated splitting if $\Sigma$ does not admit $n$-dominated
splitting for all $n>0$.

A homoclinic class $H(p, f)$ is said to be wild if there exists a $C^1$-neighborhood $\mathcal{U}$ of $f$
such that for all $g \in \mathcal{U}$, the continuation $H(P, g)$ of $H(P,f)$ does not admit dominated 
splitting. It is said to be generically wild if there exists 
a neighborhood $\mathcal{U}$ of $f$ and residual subset $\mathcal{R}$ of $\mathcal{U}$ such that 
for every $g \in \mathcal{R}$ the continuation $H(P, g)$ of $H(P,f)$ does not admit dominated 
splitting.

\subsubsection{Heterodimensional cycles}
Let $f \in \mathrm{Diff}^1(M)$ and $\Gamma$ and $\Sigma$ be two transitive hyperbolic
invariant sets for $f$.  We say $f$ has a heterodimensional cycle
associated to $\Gamma$ and $\Sigma$ if the following holds:
\begin{enumerate}
\item The indices (the dimension of the unstable manifolds) of the sets $\Gamma$ and $\Sigma$ are different.
\item The stable manifold of $\Gamma$ meets the unstable manifold of $\Sigma$ and 
the same holds for stable manifold of $\Sigma$ and the unstable manifold of $\Gamma$.
\end{enumerate}
We say the heterodimensional cycle between $\Gamma$ and $\Sigma$ is $C^1$-robust
if there exists a $C^1$-neighborhood $\mathcal{U}$ of $f$ such that for each $g \in \mathcal{U}$
the continuations $\Gamma(g)$ of $\Gamma$ and $\Sigma(g)$ of $\Sigma$
have a heterodimensional cycle.

\subsection{Background}
In the sphere of non-hyperbolic dynamics, 
one can find homoclinic classes that are robustly non-hyperbolic.
Some of them exhibit weak form of hyperbolicity, for example, 
partial hyperbolicity or dominated splitting,
while there exist homoclinic classes that robustly 
fail to have any kind of dominated splittings.
Such homoclinic classes are called wild.

It is an intriguing subject to study the effects and mechanisms of the 
absence of dominated splittings on homoclinic classes.
The study of wild homoclinic classes is commenced by Bonatti and D\'{i}az in \cite{BD}. 
They showed, under some condition on Jacobian, 
the robust absence of dominated splittings on a homoclinic class
implies the creation of complicated dynamics called universal dynamics.
In \cite{BDP}, Bonatti and D\'{i}az and Pujals proved,
$C^1$-generically, absence of dominated splittings implies the 
$C^1$-Newhouse phenomenon, i.e., locally generic coexistence of 
infinitely many sinks or sources.
These results tells us that the study of wild homoclinic classes are fertile.

In \cite{S}, inspired by \cite{ABCDW}, 
aiming to the understanding of the mechanisms of the wilderness,
the problem of the index sets of wild homoclinic classes is treated. 
Let us review the problem. For a homoclinic class $H(p)$,
its index set, denoted as $\mathrm{ind}(H(p))$, is defined to be the
collection of integers 
that appear as the index (dimension of unstable manifold) of
some hyperbolic periodic point in $H(p)$.  
In that article, it is proved that, under some assumption on Jacobian, 
$C^1$-generically, the index set of three dimensional 
wild homoclinic classes contain all possible indices, namely, $1$ and $2$.
The result sounds quite reasonable if one recalls the idea of the argument in \cite{BDP}:
The wilderness of a homoclinic class scatters its hyperbolicity to any direction.
So it is plausible one can construct periodic points with
prescribed indices by mixining the hyperbolicity with small perturbations.
 
In this article, we give an example that says this naive idea has some limitation in higher dimensionnal cases.
Let us state the precise statement of our theorem.
\begin{thm}\label{mainth}
For every four-dimensional smooth closed manifold $M$, 
there exists a diffeomorphism $f$ that satisfies the following:
There exist a hyperbolic fixed point $P$ of $f$, 
$C^1$-neighborhood $\mathcal{U}$ of $f$ 
and a residual subset $\mathcal{R}$ of $\mathcal{U}$ such that 
for every $g \in \mathcal{R}$, $H(P,g)$ does not 
admit any kind of dominated splittings and 
$\mathrm{ind}(H(P,g)) = \{2,3\}$.
\end{thm}
This theorem says that
a wild homoclinic class may have index deficiency.
More precisely, it is not always true that one can construct a saddle 
with any prescribed index, even in the sense of $C^1$-generic viewpoint. 

Let us introduce the idea of the proof.
The idea of assuring the wilderness of the homolcinic class is nothing new.
One can find a similar idea in \cite{BD}.
The novelity of our argument is the way to assure the non-existence of periodic points whose indices are not equal to $1$.
The naive idea is simple: If the diffeomorphism strongly expands some two dimensional
subspace of the tangent space at each point inside the homoclinic class,
then each three-dimensional subspace must be volume-expansive.
In particular, no periodic points inside the homoclinic class has index $1$.

Let us discuss the conditions we work on.
We work on the category of the $C^1$-diffeomorphisms of four-dimensional manifolds.
This environment is choosen only because this is enough for our purpose 
to present a example of a wild homoclinic class with index deficiency.
It is natural to wonder what happens in higher regularity and higher dimensional manifolds.
However, the complete treatise that covers general cases requires 
complicated descriptions and most of them do not seems to be essential. 
So we give up treating general cases and concentrate on this special environment. 

We want to discuss one technical issue.
Theorem \ref{mainth} says that we can construct an example of generically wild homoclinic class.
In the article such as \cite{BD}, \cite{S}, they did not 
treat generically wild homoclinic classes but wild homoclinic classes.
The difference is not so serious for the following reason:
We can re-prove many results that holds for wild homoclinic classes
starting from generically wild homoclinic classes.
For example, we can get generically wild homoclinic classes version of
the result in \cite{S} with little modification.

Yet the following question might be interesting in itself:
\begin{flushleft}
 {\bf Question.} \, Are  generically wild homoclinic classes wild?
\end{flushleft} 
For instance, it is an interesting question whether our example is wild or not. 
We would like to argue this problem in another time.

Finally, we explain the organization of this article.
In section 2, we deduce our problem to a local problem
and we describe the abstruct conditions that are sufficient
to guarantee the properties we claimed in theorem \ref{mainth}.
Section 3 is the hard part of this article. 
We construct a diffeomorphim that satisfies the 
condition listed in section 2.

\section{Sufficient condition for the theorem}

In this section, we give abstract conditions  
that assure our theorem and the proof of theorem \ref{mainth} assuming
the existence of the diffeomorphism that satisfies such conditions.

The following proposition states the sufficient conditions for theorem \ref{mainth}.
\begin{prop}\label{global}
Let $M$ be a four-dimensional smooth Riemannian manifold
and $f \in \mathrm{Diff}^1(M)$.
Suppose that $f$ satisfies all the condition below:
\begin{enumerate}
\def\labelenumi{($W$\theenumi)}
\item There are two compact sets  $A$, $B$ in $M$ such that $B \subset A$,
$f(A) \subset \mathrm{int}(A)$ and  $f(B) \subset \mathrm{int}(B)$.
\item There exist two hyperbolic fixed points $P$, $Q$ of $f$ in $C :=A \setminus B$.
\item $f$ has a heterodimensional cycle associated to $P$ and $Q$.
\item $\mathrm{ind}(P) = 3$ and let $\sigma(P), \mu_1(P), \mu_2(P), \mu_3(P)$ be the 
eigenvalues of $df(P)$ in non-decreasing order of their absolue values. 
Then $\mu_1, \mu_2$ are in $\mathbb{C} \setminus \mathbb{R}$.
\item $\mathrm{ind}(Q) = 2$ and every eigenvalue of $df(Q)$ is in $\mathbb{C} \setminus \mathbb{R}$.
\item There exists a constant $K>1$ such that $m\left( \Lambda^3(df) \right) >K$
on $C$.
In other words, $df(x)$ expands every three dimensional subspace of $TM|_x$
for all  $x \in C$ with the rate $K$.
\end{enumerate}
Then there exists a non-empty open neighborhood $\mathcal{U} \subset \mathrm{Diff}^1(M)$ of $f$ 
and a residual subset $\mathcal{R}$ of $\mathcal{U}$ such that 
for every $g \in \mathcal{R}$ the homoclinic class $H(P,g)$
does not admit dominated splittings and $\mathrm{ind}(H(P,g)) =\{ 2, 3\}$.
\end{prop}

To give the proof, we quote two lemmass.
\begin{lem}[theorem 2 in \cite{BDKS}]\label{robcycle}
Consider a diffeomorphism $f$ having a heterodimensional cycle associated to 
two hyperbolic periodic points $P$ and $Q$ with $\mathrm{ind}(P)-\mathrm{ind}(Q) = \pm1$
and $\mathrm{ind}(Q)=2$. 
Suppose $df^{\mathrm{per}(Q)}(Q)$ has a complex eigenvalue
with absolute value less than one.
Then there are diffeomorphisms $g$ arbitrarily $C^1$-close to $f$
with robust heterodimensional cycles associated to transitive hyperbolic sets 
$\Gamma(g)$ and $\Sigma(g)$ containing the continuations of $P(g)$ and
$Q(g)$ of $P$ and $Q$.
\end{lem}

\begin{lem}[\cite{BC}]\label{CRCisHC}
There exists a residual subset $\mathcal{R}_1$ in $\mathrm{Diff}^1(M)$ such that
following holds: For every $f \in \mathcal{R}_1$ and every chain recurrence class $C$ of $f$,
if it contains a hyperbolic periodic point $P$, then $H(P)$ coincides with $C$. 
\end{lem}

The following lemma is easy to prove, so the proof is left to the reader.
\begin{lem}\label{CRcyc}
Let $P$, $Q$ be hyperbolic periodic points and there exists a heterodimensional
cycle associated to two transitive hyperbolic invariant sets $\Gamma$ and $\Sigma$
and $\Gamma$ contains $P$ and $\Sigma$ contains $Q$. Then $P$ and $Q$ belong to the same 
chain recurrence class.
\end{lem}

Let us give the proof of proposition \ref{global}.

\begin{proof}[Proof of proposition \ref{global}]
Let $f$ be a diffeomorphism of a four-dimensional manifold that 
such that $f$ satisfies ($W$1)--($W$6).
By applying lemma \ref{robcycle} to $P$ and $Q$, we can take a diffeomorphism $g$, an open neighborhood
$\mathcal{U}$ of $g$ and two hyperbolic transitive invariant set $\Gamma$, $\Sigma$
such that for every $h \in \mathcal{U}$, the continuation $\Gamma(h)$ contains $P(h)$,
$\Sigma(h)$ contains $Q(h)$ and $f$ has a heterodimensional cycle assiciated to $\Gamma(h)$ and $\Sigma(h)$.
Note that by taking $\mathcal{U}$ sufficiently close to $f$, we can assume for every $h \in \mathcal{U}$ 
all the conditions ($W$1)--($W$6) holds except ($W$3).

Let us put $\mathcal{R} = \mathcal{U} \cap \mathcal{R}_1$, where $\mathcal{R}_1$
is the residual set obtained by lemma \ref{CRCisHC}. 
We show that for every $h \in \mathcal{R}$ conclusions of proposition \ref{global} hold.
First we show $H(P,h)$ is not dominated. By $h \in \mathcal{U}$, we know
there exists a heterodimensional cycle associated to $\Gamma(h)$ and $\Sigma(h)$
with $P \in \Gamma(h)$ and $Q \in \Sigma(h)$.
Then lemma \ref{CRcyc} implies that $P(h)$ and $Q(h)$ belong to the same
chain recurrence set. For $h \in \mathcal{R}_1$, this chain recurrence set 
coincides with $H(P,h)$ and simultaneously $H(Q, h)$, in particular $H(P, h) = H(Q, h)$.
$H(P, h)$ does not admit dominated splitting of the form $E \oplus F$ where $\dim E =1$ or $3$,
because $P(h)$ has two contracting complex eigenvalues and expanding eigenvalues.
We can also see  does not admit dominated splitting $E \oplus F$ where $\dim E =2$,
because $H(P,h)$ contains $Q(h)$ and ($W$5) holds.
Thus we know $H(P,h)$ does not admit any kind of dominated splittings.

Let us see $\mathrm{ind}(H(P,h)) = \{2, 3\}$. Since $H(P,h)$ contains $P$ and $Q$,
we know $\mathrm{ind}(H(P,h))$ contains $2$ and $3$. We need to show
$\mathrm{ind}(H(P,h))$ does not contain $1$.
To check this, we prove $H(P, h) \subset C$.
Indeed, if there is a hyperbolic periodic point $R \in H(P,h)$, 
then by assumption ($W$6) every iteration of $dh$ expands every three dimensional subspace
of the tangent space. So the indices of any periodic points in $H(P, h)$ cannot be $1$.

For the first step to see $H(P,h) \subset C$, we show $W^s(P,h) \ti W^u(P,h) \subset C$. 
Let us take $x \in W^s(P,h) \ti W^u(P,h)$.
Since $x \in W^u(P,h)$, there exists $n>0$ such that $h^{-n}(x) \in A$.
So we get $x \in h^n(A)$. Since $A$ satisfies $h(A) \subset \mathrm{int}(A)$, 
we know $h^n(A) \subset A$ for $n >0$ and this means $x \in A$. 
Now we show $x \not\in B $. 
Since $B$ satisfies $B \subset \mathrm{int}(B)$, we get $h^n(B) \subset B$.
If $x \in B$, then for all $n>0$ we have $h^n(x) \in B$, but since $x \in W^s(P,h)$
we know $f^n(x)$ converges to $P(h)$ as $n \to \infty$. This contradicts the 
fact  $P(h) \not \in B$. Thus we have $W^s(P,h) \ti W^u(P,h)$.

In the following we show $H(P, h) \subset C$.
Take $y \in H(P, h)$. By definition there exists a sequence
$(x_n) \subset W^s(P,h) \ti W^u(P,h)$ such that $x_n$ converges to $y$ as $n \to \infty$.
Since $W^s(P,h) \ti W^u(P,h) \subset A$
and $A$ is compact, we know $H(P, h) \subset A$, in particular $ y \in A$.
Let us prove $y \not\in B$. Suppose that $y \in B$ and 
let us consider the point $h(y)$. This point belongs to $\mathrm{int}(B)$ 
and it means there exists an neighborhood $U$ of $h(y)$ that is contained in $\mathrm{int}(B)$.
Since $h$ is continuous, we know the sequence $(h(x_n))$ converges to $h(y)$.
Thus we know there exists $N$ such that $h(x_N)$ belongs to $U$, in particular to $B$. 
Since $x_N \in W^s(P,h) \ti W^u(P,h)$, we have $h(x_N)\in W^s(P,h) \ti W^u(P,h)$.
This is a contradiction, because we already proved $W^s(P,h) \ti W^u(P,h)$ is disjoint from $B$.
Therefore we proved $H(P, h) \subset C$ and finished the proof of proposition \ref{global}.
\end{proof}

Let us see how to prove theorem \ref{mainth}.
We prove the following proposition in section \ref{example}.
\begin{prop}\label{local}
There exists a diffeomorphism of $f \in \mathrm{Diff}(\mathbb{R}^4)$ such that following holds:
\begin{enumerate}
\def\labelenumi{($w$\theenumi)}
\item The support of $f$ is compact.
\item There are two compact sets  $A$, $B \subset \mathbb{R}^4$ with $B \subset A$, 
$f(A) \subset \mathrm{int}(A)$ and  $f(B) \subset \mathrm{int}(B)$.
\item There exist two fixed points $P$, $Q$ of $f$ in $C:=A \setminus B$.
\item $f$ has a heterodimensional cycle associated to $P$ and $Q$.
\item $\mathrm{ind}(P) = 3$ and let $\sigma(P), \mu_1(P), \mu_2(P), \mu_3(P)$ be the 
eigenvalues of $df(P)$ in non-decreasing order of their absolue values. 
Then $\mu_1(P), \mu_2(P)$ are in $\mathbb{C} \setminus \mathbb{R}$.
\item $\mathrm{ind}(Q) = 2$ and every eigenvalue of $df(Q)$ is in $\mathbb{C} \setminus \mathbb{R}$.
\item There exists a constant $K>1$ such that $m\left( \Lambda^3(df) \right)>K$
on $C$ (we furnish $\mathbb{R}^4$ with standard Riemannian structure). 
In other words, $df(x)$ expands every three-dimensional subspace of 
$T\mathbb{R}^4|_x$ at any $x \in C$ with the rate $K$.
\end{enumerate}
\end{prop}

Let us prove theorem \ref{mainth} by using proposition \ref{global} and proposition \ref{local}.
\begin{proof}[Proof of theorem \ref{mainth}.]
Let us take a diffeomorphism $f$ of $\mathbb{R}^4$ which satisfies 
($w$1)--($w$7) by using proposition \ref{local}.
Given four-dimensional compact manifold $M$, 
we take a point $x \in M$ and a coordinate chart $\phi :U \to \mathbb{R}^4 $  around $x$.
By changing the coordinate if necessary, we can assume $\phi(U)$ contains 
the support of $f$.
Then we can define a diffeomorphism $F \in \mathrm{Diff}^1(M)$
as follows: $F(x) = \phi^{-1} \circ f \circ \phi(x)$ if $x \in U$, otherwise $F(x) = x$.
Let us denote the standard Riemannian metric on $\mathbb{R}^4$ as $g$.
We can construct a Riemannian metric $\tilde{g}$
on $M$ that coincides with the pullback of $g$ by $\phi$ at every point in $\phi^{-1}(A)$
with the partition of unity.
Now by applying proposition \ref{global} to the triplet $(M,F,\tilde{g})$,
we can find the diffeomorphism and its open neighborhood
that satisfy the conclusion of theorem \ref{mainth}.
\end{proof}
Thus what is left to us is the proof of proposition \ref{local}.
Let us give its proof in the next section.

\section{Construction of the diffeomorphism}\label{example}
In this section, we give the proof of proposition \ref{local}.
\subsection{Notations and sketch of the proof}
In this subsection, we prepare some notations 
and provide the naive idea of the proof of proposition \ref{local}.
\begin{figure}[t]
 \begin{center}
  \input{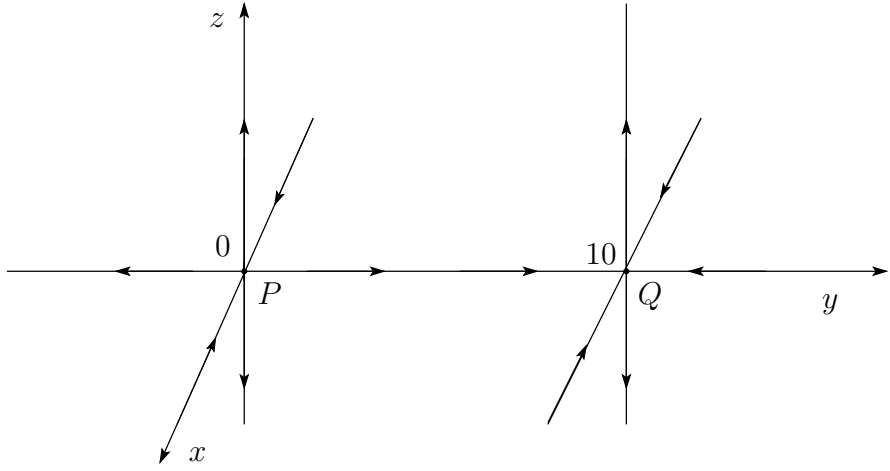}
 \end{center}
 \caption{Dynamics of $\Phi$.}
 \label{fig:phi}
\end{figure}

We identify the point $X$ of $\mathbb{R}^4$ with the vector
that starts from the origin and ends at $X$. Under this identification, 
we define the addition between any two points of $\mathbb{R}^4$ and multiplication
by real numbers.
We denote the standard Euclidean distance of $\mathbb{R}^4$ as $d(\,\cdot\, , \,\cdot\, )$.
We sometimes use the $\ell^{\infty}$-distance of $\mathbb{R}^4$ and denote it 
as $d_1(\,\cdot\, , \,\cdot\, )$. More presicely, given 
$X=(x_1, x_2, x_3, x_4), Y=(y_1, y_2, y_3, y_4) \in  \mathbb{R}^4$, we define
\[
d(X, Y) = \bigg( \, \sum_{i =1,\ldots ,4} (x_i - y_i)^2 \bigg)^{1/2},  \quad  d_1(X, Y) = \max_{i =1,\ldots ,4} |x_i - y_i|.
\]
We name some points and subsets of $\mathbb{R}^4$.
First, we name some points as follows:
We define three subsets in $\mathbb{R}^4$ as follows: 
\begin{eqnarray*}
\ell_1  & := \{ (10 +x , 0, 0, 0)  \,\, | \,\, |x| < 0.2 \}, \\
\ell_2  & := \{ (0, 5+y, 0, 0)      \,\, | \,\, |y| < 1 \}, \\
\varpi  & := \{ (0, 10, z, w)        \,\, | \,\, |z|, |w| < 0.2 \}.
\end{eqnarray*}

For $X \in \mathbb{R}^4$ and $l >0$,
by $\mathcal{B}(X, l)$ we denote four-dimensional cube with edges of length $2l$ centered at $X$.
More presicely, we put
\[
\mathcal{B}(X, l) :=\{ Z \in \mathbb{R}^4 \mid d_1(X, Z) \leq l \}.
\]
With this notation, we define as follows:
For $X, Y \in \mathbb{R}^4$ and $l < 2d(X, Y)$,
by $\mathcal{C}(X,Y, l)$, we denote four-dimensional box defined as the collection of points
whose $\ell^{\infty}$-distance from the segment joining $X$ and $Y$ is less than $l$. 
More presicely, we put 
\[\mathcal{C}(X,Y, l) := \bigcup_{0 \leq t \leq 1} \mathcal{B}(tX+(1-t)Y, l). \]
With this notation, we define 
\[
D := \mathcal{C}(C_1, C_{2} , 1) \cup \mathcal{C}(C_{2}, C_{3} , 0.7) \cup 
\mathcal{C}(C_{3}, C_{4} , 0.4).
\]
Throughout this section, $\lambda$ denotes a real number greater than $20$. 
We define as follows:
Note that these sets depend on the value of $\lambda$ 
while the points and sets defined previously are independent of $\lambda$.

We give the idea of the proof of proposition \ref{local}. We construct two diffeomorphisms
$\Phi, \Upsilon$ of $\mathbb{R}^4$ such that the composition $\Upsilon \circ \Phi$
gives a diffeomorphism that satisfies conditions ($w$1)--($w$7) of proposition \ref{local}
except ($w$4).
Let us briefly see the role of each diffeomorphism.
$\Phi$ is a fundamental diffeomorphism. $\Phi$ has two fixed points $P$ and $Q$
whose eigenvalues are as is described in ($w$5)--($w$6)
and there is an intersection between $W^u(P)$ and $W^s(Q)$ (see Figure \ref{fig:phi}).
Furthermore, $\Phi$ is constructed so that once a point escapes, then 
it never returns close to $P$ and $Q$ (this corresponds to the condition ($w$2)).

To obtain condition ($w$4), we need to connect $W^u(Q)$ and $W^s(P)$.
The diffeomorphism $\Upsilon$ is utilized for this purpose.
$\Upsilon$ pushes $W^u(Q, \Phi)$ so that 
it has intersection with $W^s(P)$. 
We need to guaranatee this perturbation give little effect on 
the connection between $W^u(P)$ and $W^s(Q)$ and the whole structure of the dynamics.
So the push makes a detour (see see Figure \ref{fig:Upsilon}).

The diffeomorphism $\Upsilon$ is constructed 
independent of the value of $\lambda$. After the construction of $\Upsilon$, 
we take $\lambda$ sufficiently large so that the effect of  $d\Upsilon$
becomes ignorable when we check the condition ($w$7) 
for $\Upsilon \circ \Phi$. 

\begin{figure}[t]
 \begin{center}
  \input{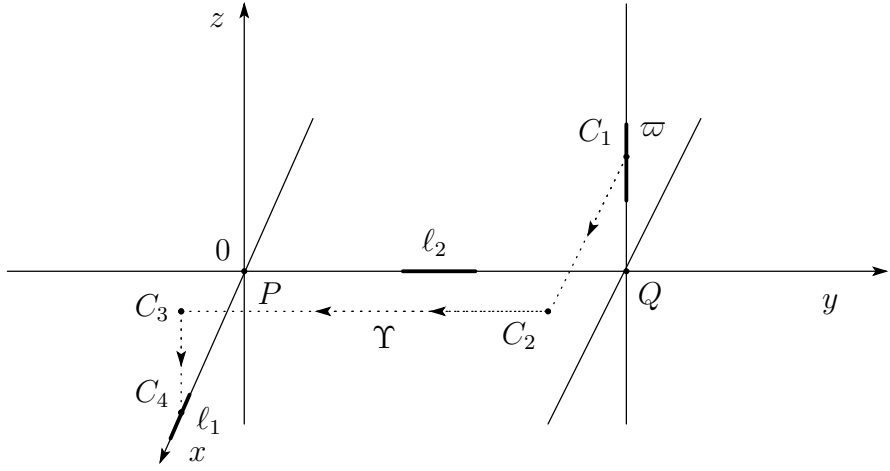}
 \end{center}
 \caption{The role of  $\Upsilon$.}
 \label{fig:Upsilon}
\end{figure}

Finally, we prepare a lemma that we frequently use throughout this section.
We omit its proof, since it can be found in many textbooks.
\begin{lem}\label{cptsupp}
Given two closed interval $[a, b] \subset [c, d]$ ($a, b, c, d$ may be $\pm\infty$), 
there exists a $C^{\infty}$-function $\rho[a,b,c,d](t) : \mathbb{R} \to \mathbb{R}$ that satisfies the following conditions:
\begin{itemize}
\item For all $t \in \mathbb{R}, 0 \leq \rho[a,b,c,d](t) \leq 1$.
\item If $t \in [a, b]$, then $\rho[a,b,c,d](t) = 1$.
\item If $t  \not\in [c, d]$, then $\rho[a,b,c,d](t) = 0$.
\end{itemize}
\end{lem}

\subsection{Construction of $\Phi$}
In this section, we construct the diffeomorphism $\Phi$.
\begin{prop}\label{Phi}
There exists a diffeomorphism $\Phi : \mathbb{R}^4 \to \mathbb{R}^4$ 
that has the following properties:
\begin{enumerate}
\def\labelenumi{($\Phi$\theenumi)}
\item The support of $\Phi$ is contained in $[-\lambda^3-1, \lambda^3+1]^4$. 
\item $\Phi(A) \subset \mathrm{int}(A)$ and $\Phi(B) \subset \mathrm{int}(B')$.
\item $P$, $Q$ are fixed points of $\Phi$.
\item $\mathrm{ind}(P) = 3$ and let $\sigma(P), \mu_1(P), \mu_2(P), \mu_3(P)$ be the 
eigenvalues of $df(P)$ in non-decreasing order of their absolue values.
Then $\mu_1(P), \mu_2(P)$ are in $\mathbb{C} \setminus \mathbb{R}$.
\item $\ell_1 \subset W^s(P, \Phi)$ and $ \cup _{n \geq 1} \Phi^{n}(\ell_1)  \cap D = \emptyset$.
\item $\ell_2 \subset W^u(P, \Phi)$ and $ \cup _{n \geq 0} \Phi^{-n}(\ell_2)  \cap D = \emptyset$.
\item $\mathrm{ind}(Q) = 2$ and every eigenvalue of $d\Phi(Q)$ is in $\mathbb{C} \setminus \mathbb{R}$. 
\item $\ell_2 \subset W^s(Q, \Phi)$ and $ \cup _{n \geq 0} \Phi^{n}(\ell_2)  \cap D = \emptyset$.
\item $\varpi \subset W^u(Q, \Phi)$ and $ \cup _{n \geq 1} \Phi^{-n}(\varpi)  \cap D = \emptyset$.
\item There exists a constant $c_{\Phi}>0$ (independent of $\lambda$) 
such that the inequality $m \left( \Lambda^3(d\Phi)(X) \right) > c_{\Phi}\lambda$ holds for every $X \in C$.
\end{enumerate}
\end{prop}

To prove proposition \ref{Phi}, we create auxiliary diffeomorphims $F$, $G$ and $H$ of $\mathbb{R}$
and $\Theta$ of $\mathbb{R}^4$.
We first list up the properties of these diffeomorphisms and prove such diffeomorphims exist.
After that we give the proof of proposition \ref{Phi}.

\begin{lem}\label{diffF}
\def\labelenumi{($F$\theenumi)}
There exists a $C^{\infty}$-diffeomorphism $F$ of $\mathbb{R}$ that satisfies the following:
\begin{enumerate}
\def\labelenumi{($F$\theenumi)}
\item The support of $F$ is contained in $[-\lambda^3, \lambda^3]$.
\item For any $x \in [-\lambda^2, \lambda^2] \setminus \{0\}$, the inequality  $0 < F(x) /x < 1/9$ holds.
\item There is a constant $c_1 >0$ (independent of $\lambda$) such that the inequality
$\min_{x \in [-\lambda ^2 , \lambda ^2]} F'(x) > c_1$ holds.
\end{enumerate}
\end{lem}

\begin{proof}
Let us consider the vector field on $\mathbb{R}$ given by $\dot{x}(t) = f(x)$, 
where $f(x)$ is a $C^{\infty}$-function on $\mathbb{R}$ that has the following properties:
\begin{itemize}
\item If  $|x| < \lambda^3-2,  f(x) = -\log(10)x$.
\item If  $|x| > \lambda^3 -1, f(x) =0$.
\end{itemize}
We can construct such $f$ as follows:
\[
f(x) = - \log(10)x\rho[-(\lambda^3-2), \lambda^3-2,-(\lambda^3-1), \lambda^3-1](x).
\]
Let us take the time-$1$ map of this vector field and denote it as $F(x)$
(we can consider the time-$1$ map for all $x \in \mathbb{R}$ 
since $f$ has  compact support and Lipschitz continuous).
Then it is not difficult to check $F(x)$ satisfies all the properties ($F$1)--($F$3).
So the detail is left to the reader.
\end{proof}

\begin{rem}
By the conditions ($F$1)--($F$3), we can prove $F$ satisfies the followoing conditions:
\begin{enumerate}
\def\labelenumi{($F$\theenumi)}
\setcounter{enumi}{3}
\item $0$ is an attracting fixed point of $F$. 
\item $F([-\lambda^2, \lambda^2]) \subset (-\lambda^2, \lambda^2)$.
\end{enumerate}
\end{rem}

\begin{lem}\label{diffG}
\def\labelenumi{($G$\theenumi)}
There exists a $C^{\infty}$-diffeomorphism $G$ of $\mathbb{R}$ has the following properties:
\begin{enumerate}
\item The support of $G$ is contained in $[-\lambda^3, \lambda^3]$.
\item $G([-\lambda^2, \lambda^2]) \subset (-\lambda^2, \lambda^2)$.
\item For any $x \in [-1/100, 1/100] \setminus \{0\}$, the inequality  $ 9 < G(x) /x  < 10$ holds.
\item For any $x \in [-1/100, 1/100] \setminus \{0\}$, the inequality  $  0 < (G(x +10) -10 ) /x < 1/10$ holds.
\item For any $x \in (0, 10)$, $\lim_{n \to \infty} G^{-n}(x) = 0$ and $\lim_{n \to \infty} G^{n}(x) = 10$.
\item There is a constant $c_2 >0$ (independent of $\lambda$) such that the inequality 
$\min_{y \in [-\lambda ^2 , \lambda ^2]} G'(y) > c_2$ holds.
\end{enumerate}
\end{lem}

\begin{proof}
Let us consider the vector field $\dot{y}(t) = g(y)$, where $g(y)$ is a $C^{\infty}$-function 
on $\mathbb{R}$ satisfying the following properties:
\begin{itemize}
\item If  $|y| < \lambda^3-2,  g(y) = -\log(10)/100(y^3 -100y)$.
\item If  $|y| > \lambda^3 -1, g(y) =0$.
\end{itemize}
We can construct such $g$ as follows:
\[
g(y) = - \log(10)/100(y^3 -100y)\rho[-(\lambda^3-2), \lambda^3-2,-(\lambda^3-1), \lambda^3-1](y).
\]
Let us take the time-$1$ map of this vector field and denote it as $G(x)$
(we can consider the time-$1$ map for all $x \in \mathbb{R}$ 
since $g$ has  compact support and Lipschitz continuous).
Then it is not difficult to check $G(x)$ satisfies all the property ($G$1)--($G$6).
So the detail is left to the reader.
\end{proof}

\begin{rem}
By the conditions ($G$1)--($G$6), we can prove $G$ satisfies the followoing conditions:
\begin{enumerate}
\def\labelenumi{($G$\theenumi)}
\setcounter{enumi}{6}
\item $0$ is a repelling fixed point of $G$. 
\item $10$ is an attracting fixed point of $G$.
\end{enumerate}
\end{rem}

\begin{lem}\label{diffH}
There exists a $C^{\infty}$-diffeomorphism $H$ on $\mathbb{R}$ which has the following properties:
\begin{enumerate}
\def\labelenumi{($H$\theenumi)}
\item For every $z$, $H(-z) =-H(z)$. 
\item The support of $H(z)$ is contained in $[-\lambda^3, \lambda^3]$.
\item $H([-\lambda^2, \lambda^2]) \subset   (-\lambda^2, \lambda^2)$.
\item $H([ 1/2, \lambda^2]) \subset   (7, \lambda^2)$.
\item For $0 \leq z \leq 1$, $H(z) = \lambda z$. In particular, $0$ is a repelling fixed point of $H$.
\end{enumerate}
\end{lem}

\begin{proof}
First, we construct a $C^{\infty}$-function $h(t)$ on $\mathbb{R}_{\geq 0}$ that satisfies the following:
\begin{enumerate}
\def\labelenumi{($h$\theenumi)}
\item For all $t\geq 0, h(t) >0$. 
\item For $0 \leq t \leq 1$, $h(t) = \lambda$.
\item $ \int_1^2 h(t)\,dt = 1$.
\item For $2 \leq t \leq \lambda^2, h(t) = \lambda^{-2}$.
\item $ \int_0^{\lambda^3-1} h(t) \,dt= \lambda^3-1$
\item For $t \geq \lambda^3-1, h(t) = 1$.
\end{enumerate}
Then by putting $H(z) := \int_0^zh(t)dt$ for $z \geq 0$ and $H(z) = - H(-z)$ for $z \leq 0$
we construct a map $H(z)$.
Let us discuss how to construct $h(t)$.
We put
\begin{eqnarray*}
\rho_1(t) &= \left( \lambda - \lambda^{-2} \right)\rho[-\infty,1, -\infty, 1+\lambda^{-2}](t), \\
\rho_2(t;\alpha) &=  \alpha \rho[ 7/5, 8/5, 6/5, 9/5](t), \\
\rho_3(t) &= \left( 1 - \lambda^{-2} \right)\rho[\lambda^3-1, +\infty, \lambda^3-2, +\infty](t).\\
\rho_4(t; \beta) &= \beta\rho[\lambda^3-8/5, \lambda^3-7/5, \lambda^3-9/5, \lambda^3-6/5, ](t).
\end{eqnarray*}
Take a function 
$\eta(t;\alpha, \beta) = \rho_1(t) + \rho_2(t;\alpha) + \rho_3(t) + \rho_4(t;\beta) + \lambda ^{-2}$.
Note that for all $\alpha ,\beta \geq 0$ the function $\eta(t;\alpha, \beta)$ 
satisfies the conditions ($h$1)--($h$6) except ($h$3) and ($h$5).
We show there exist  positive real numbers $\alpha_0$ and $\beta_0$ such that ($h$3), ($h$5) hold for 
$\eta(t;\alpha_0, \beta_0)$.

First, let us see how to take $\alpha_0$ such that ($h$3) holds for $\eta(t;\alpha_0, \beta)$.
Let us put $\Xi (\alpha) = \int_1^2 \eta(t;\alpha, \beta) dt$. 
Then by definition $\Xi(\alpha)$ is  independent of $\beta$, continuous, 
monotonously increasing and $\Xi(\alpha) \to \infty$ as $\alpha \to +\infty$.
Moreover, $\Xi(0) \leq \lambda \cdot \lambda^{-2} + 1 \cdot \lambda^{-2} < 1$ 
(remember that we assume $\lambda >100$).
Thus by intermediate value theorem we get $\alpha_0$ with $\Xi(\alpha_0) =1$.
Then $\eta(t;\alpha_0, \beta)$ satisfies ($h$3).
In the similar way we can find $\beta_0$ so that ($h$5) holds for $\eta(t;\alpha_0, \beta_0)$.

It is easy to see that $H$ is a $C^{\infty}$-diffeomorphism and enjoys ($H$1), ($H$2) and ($H$5).
Let us check ($H$3) and ($H$4).
We have $H(1/2) = \lambda /2 > 7$ and
$H(\lambda ^2)$ can be estimated as follows:
\[
H(\lambda ^2) = H(2) + (\lambda^2-2)\lambda^{-2} = \lambda +1 +1 - 2/\lambda^2 < \lambda  - 2 < \lambda^2.
\]
Hence we have proved ($H$3) and ($H$4).
\end{proof}

\begin{lem}
There exists a $C^{\infty}$-diffeomorphism $\Theta$ of $\mathbb{R}^4$ that has the following properties:
\begin{enumerate}
\def\labelenumi{($\Theta$\theenumi)}
\item The support of $\Theta$ is contained in $B_l(P) \cup B_l(Q)$.
\item For $X = (x, y, z, w) \in B_s(P), \Theta(x,y,z,w) = (x,-z,y,w)$. In particular, $P$ is a fixed point of $\Theta$.
\item $\Theta$ fixes every point in $x$-axis. More presicely, for $X = (x, 0, 0, 0),  \Theta(X) = X$.
\item For any $X \in B_l(P)$,  $\Theta$ preserves the $d$-distance between $P$ and $X$.
More precisely, $d(P, \Theta(X) ) =d(P, X)$.
\item In $B_l(P)$, $\Theta$ preserves the $yz$-plane. More presicely, 
for $X=(0, y, z, 0) \in B_l(P)$, the $x$-coordinate and $w$-coordinate of $\Theta(X)$ are $0$.
\item For $(x, y+10, z, w) \in B_s(Q), \Theta(x,y+10,z,w) = (-y, x+10, -w, z)$. In particular, $Q$ is a fixed point of $\Theta$.
\item For any $X \in B_l(Q)$, $\Theta$ preserves the $d$-distance between $Q$ and $X$.
\item In $B_l(Q)$, $\Theta$ preserves the $xy$-plane.
More presicely, for $X=(x, y, 0, 0) \in B_l(Q)$, the $z$-coordinate and $w$-coordinate of $\Theta(X)$ are $0$.
\item In $B_l(Q)$, $\Theta$ preserves the plane that passes $Q$ and parallel to $zw$-plane. 
More presicely, for $X=(0, 10, z, w) \in B_l(Q)$, the $x$-coordinate of $\Theta(X)$ is $0$ and the $y$-coordinate  of $\Theta(X)$ is $10$.
\end{enumerate}
\end{lem}

 \begin{proof}
 First we define three functions $\rho_4(t), \omega_1(X), \omega_2(X)$ as follows: 
 \begin{eqnarray*}
 \rho_1(t) &= \rho[-1/300, 1/300, -1/200, 1/200](t), \\
 \omega_1(x, y, z, w) &= (\pi/2)\rho_1(x)  \rho_1 (y) \rho_1(z)\rho_1(w), \\
 \omega_2(x, y, z, w) &= (\pi/2)\rho_1(x)  \rho_1 (y-10) \rho_1(z)\rho_1(w).
 \end{eqnarray*}
We also define a map $R[\alpha]:\mathbb{R}^2 \to \mathbb{R}^2$ 
to be the rotation of angle $\alpha$, more precisely, for  $(x, y) \in \mathbb{R}^2$
we put
\[
R[\alpha](x, y) :=( \cos(\alpha)x -\sin(\alpha)y, \, \sin(\alpha)x +\cos(\alpha)y ).
\]
Then we define $\Theta$ as follows: 
\begin{itemize}
\item For $X = (x, y, z, w) \in B_l(P)$,  $\Theta (X) = (x, R[\omega_1(X)](y,z) , w)$.
\item For $X = (x, y+10, z, w) \in B_l(Q)$,  
\[\Theta (X) = (R[\omega_2(X)](x,y), R[\omega_2(X)](z,w) ) + Q.\]
\item Otherwise, $\Theta(X) = X$.  
\end{itemize}
Then $\Theta$ is a $C^{\infty}$-diffeomorphism and has all the properties ($\Theta$1)--($\Theta$9).
\end{proof}

\begin{lem}
There is a $C^{\infty}$-diffeomorphism $\Psi$ of $\mathbb{R}^4$ that has the following properties:
\begin{enumerate}
\def\labelenumi{($\Psi$\theenumi)}
\item The support of $\Psi$ is contained in $[-\lambda^3-1, \lambda^3+1]^4$.
\item For $(x, y, z, w) \in [-\lambda^3, \lambda^3]^4$, $\Psi(x, y, z, w) = (F(x), G(y), H(z), H(w))$.  
\end{enumerate}
\end{lem}

\begin{proof}
We define some functions $\rho_1$ on $\mathbb{R}$ and $R_i \,\,(i =1,\ldots,4)$ on $\mathbb{R}^4$ as follows:
\begin{eqnarray*}
\rho_1(t) := \rho[-\lambda^3, \lambda^3, -\lambda^3-1, \lambda^3+ 1](t),    \\
R_1(X)  := \rho_1(x_2)\rho_1(x_3)\rho_1(x_4), \quad
R_2(X)  := \rho_1(x_1)\rho_1(x_3)\rho_1(x_4),   \\
R_3(X)  := \rho_1(x_1)\rho_1(x_2)\rho_1(x_4),  \quad
R_4(X)  := \rho_1(x_1)\rho_1(x_2)\rho_1(x_3). 
\end{eqnarray*}
where $X=(x_1, x_2, x_3, x_4)$. Then we define a map $\Psi : \mathbb{R}^4 \to \mathbb{R}^4$ by
\[
\Psi_i(X) :=  R_i(X)F_i(x_i) + (1-R_i(X))x_i,  \quad \mbox{for} \,\, i=1, 2, 3, 4,
\]
where $\Psi_i(X)$ denotes the $i$-th coordinate of $\Psi(X)$
and we put $F_1 :=F$, $F_2=G$ and $F_3=F_4=H$ as the matter of convenience.

It is easy to see that $\Psi$ is a $C^{\infty}$-map and satisfies ($\Psi$1) and ($\Psi$2).
The perplexing part is to confirm that $\Psi$ is a diffeomorphism.
So let us check it.

We put $I_0 := [ -\lambda^3, \lambda^3]$, $I_1 := \mathbb{R} \setminus I_0$,
$S := \{(\sigma_i)_{i=1,\ldots,4} \mid \sigma_i =0,1 \}$ and for any $(\sigma_i) \in S$
we put $I[(\sigma_i)] : =  \prod_{i=1,\ldots,4}I_{\sigma_i}$.
Then divide $\mathbb{R}^4$ into $16$ subsets as follows:
\[
\mathbb{R}^4 =  \coprod_{(\sigma_i) \in S} I[(\sigma_i)].
\]
We claim that every restriction of $\Psi$ to $I[(\sigma_i)]$ is a diffeomorphism.
Fix $(\sigma_i) \in S$ and put $j := \sum \sigma_i$. We define
\[
P[(\sigma_i)] :=\{ (\sigma_ix_i) \in \mathbb{R}^4 \mid x_i \in I_1 \},
\]
and for $Y = (\sigma_iy_i) \in P[(\sigma_i)]$, 
\[
S([(\sigma_i)], Y) := \{ Y + ((1-\sigma_i)z_i) \mid z_i \in I_0 \}.
\]
Then we have the following decomposition:
\[
I[(\sigma_i)] = \coprod_{Y \in P[(\sigma_i)]} S([(\sigma_i)], Y). 
\]
Intuitively speaking, we divided $I[(\sigma_i)]$ into 
$(4-j)$-dimensional cubes that are parametrized by $j$-dimensional parameters.

First, let us investigate the behavior of the restriction of $\Psi$ to $I[(\sigma_i)]$.
We prove $\Psi_k|_{I[(\sigma_i)]}(X) = x_k$ for $k$ that satisfies $\sigma_k =1$. 
One important observation is that $F_k$ are the identity map on $I_1$
by ($F$1), ($G$1) and ($H$2).
So, if $\sigma_k = 1$, then $x_k \in I_1$. 
This means $F_k(x_k) =x_k$ and accordingly 
we have $\Psi_k(X) =x_k$ independent of the value of $R_k(X)$.

Second, we investigate the behavior of the restriction of $\Psi$ to $S([(\sigma_i)], Y)$.
We show $\Psi_k|_{S([(\sigma_i)], Y)}(X) = \tilde{F}_k(x_k)$ for $k$ that satisfies $\sigma_k =0$,
where $\tilde{F}_k(x_i)$ is some diffeomorphism of $I_0$. 
The key point is that the change of $X$ in $S([(\sigma_i)], Y)$
never varies $R_k(X)$. We can see this as follows: For $l$ that satisfies $\sigma_l =1$,
the $l$-th coordinate of $X$ is fixed, and for $l$ that satisfies $\sigma_l =0$,
$l$-th coordinate of $X$ is in $I_0$, hence the change of $X$ in $S([(\sigma_i)], Y)$ give
no effect on $R_k(X)$.
Now the formula $\Psi_k|_{S([(\sigma_i)], Y)}(X) = R_k(X)F_k(x_k) + (1-R_k(X))x_k$
tells us that the RHS gives a diffeomorphism of $I_0$, since it is a convex combination of 
diffeomorphisms of $I_0$.
In particular, $\Psi|_{S([(\sigma_i)], Y)}$ is a diffeomorphism of $S([(\sigma_i)], Y)$.

Thus we know $\Psi|_{I[(\sigma_i)]}$ is a diffeomorphism when it is restricted to the cube $S([(\sigma_i)], Y)$,
behaves as the identity map to $P[(\sigma_i)]$ direction and these 
two facts tells us $\Psi|_{I[(\sigma_i)]}$ is a diffeomorphism of $I[(\sigma_i)]$.
\end{proof}

Now, let us give the proof of proposition \ref{Phi}.
\begin{proof}[Proof of proposition \ref{Phi}.]
We put $\Phi := \Theta \circ \Psi$.
From the properties of $F$, $G$, $H$, $\Theta$ and $\Psi$, we can see $\Phi$
satisfies ($\Phi$1)--($\Phi$10). Indeed,
\begin{itemize}
\item ($\Phi$1) follows from ($\Psi$1) and ($\Theta$1).  
\item ($\Phi$2) follows from ($F$5), ($G$2), ($H$3), ($H$4) and ($\Theta$1).  
\item ($\Phi$3) follows from ($F$4), ($G$7), ($G$8), ($H$5), ($\Theta$2) and ($\Theta$6).  
\item ($\Phi$4) follows from ($F$2), ($G$3), ($H$5) and ($\Theta$2).  
\item ($\Phi$5) follows from ($F$2), ($\Theta$1)  and ($\Theta$3).  
\item ($\Phi$6) follows from ($F$2), ($G$3), ($G$5), ($\Theta$1), ($\Theta$4) and ($\Theta$5).  
\item ($\Phi$7) follows from ($F$4), ($G$8), ($H$5)  and ($\Theta$6).  
\item ($\Phi$8) follows from ($F$2), ($G$4), ($G$5), ($\Theta$1), ($\Theta$7) and ($\Theta$8).  
\item ($\Phi$9) follows from ($H$5), ($\Theta$1), ($\Theta$7), and ($\Theta$9).  
\end{itemize}
Let us check ($\Phi$10). We see the action of $\Lambda^3(d\Psi)(X)$ (we put $X = (x,y,z,w)$)
to the orthonormal basis  $ \langle e_1 \wedge e_2 \wedge e_3,  e_1 \wedge e_2 \wedge e_4, e_1 \wedge e_3 \wedge e_4, e_2 \wedge e_3 \wedge e_4 \rangle$,
where $\langle e_i\rangle$ denotes the standard orthonormal basis of $T\mathbb{R}^4|_X$.
\begin{eqnarray*}
\Lambda^3(d\Psi)(X)(e_1 \wedge e_2 \wedge e_3) &= F'(x)G'(y)H'(z)   e_1 \wedge e_2 \wedge e_3      \\
\Lambda^3(d\Psi)(X)(e_1 \wedge e_2 \wedge e_4) &= F'(x)G'(y)H'(w)  e_1 \wedge e_2 \wedge e_4      \\
\Lambda^3(d\Psi)(X)(e_1 \wedge e_3 \wedge e_4) &= F'(x)H'(z)H'(w)  e_1 \wedge e_3 \wedge e_4      \\
\Lambda^3(d\Psi)(X)(e_2 \wedge e_3 \wedge e_4) &= G'(y)H'(z)H'(w)  e_2 \wedge e_3 \wedge e_4
\end{eqnarray*}
By this calculation, we know $m \left(\Lambda^3(d\Psi)(X)\right)$ is given 
as the minimum of the lengths of these four vectors. 
Let us fix a real number $c>0$ such that
\[
 \min_{(x,y) \in [-\lambda^2, \lambda^2]^2} \left\{ |F'(x)G'(y)|,  |F'(x)|, |G'(y)| \right\} > c. 
\]
We can take such $c$ for ($F$3) and ($G$6).
Note that $c$ can be choosen independent of $\lambda$.
Since for any $(x, y, z, w) \in C$, $H'(z)=H'(w)= \lambda$,
we have proved the lengths of these four vectors are bounded below by $c \lambda$
and $m \left(\Lambda^3(d\Psi)(X)\right) > c \lambda$ for $X \in C$.

Let us put $c_{\Theta} := \min_{x \in \mathbb{R}^4 } m(\Lambda^3(d\Theta)(x))$.
Since $\Theta$ is a diffeomorphism and have compact support, $c_{\Theta}$ is a positive number.
Now, for any $X \in C$, the inequality
\[
m(\Lambda^3(d\Phi)(X)) \geq  \min_{x \in \mathbb{R}^4 } m(\Lambda^3(d\Theta)(x)) \cdot 
m \left(\Lambda^3(d\Psi)(X)\right)  = c_\Theta c\lambda
\]
holds and this means we have ($\Phi$10) for $c_{\Phi} =c_\Theta c$.

So the proof is completed.
\end{proof}

\subsection{Construction of $\Upsilon$.}
\begin{prop}
There exists a $C^{\infty}$-diffeomorphism $\Upsilon$ of $\mathbb{R}^4$ that has the following properties:
\begin{enumerate}
\def\labelenumi{($\Upsilon$\theenumi)}
\item The support of $\Upsilon$ is contained in $D$.
\item For $X \in \mathcal{B}(C_1, 0.2),  \Upsilon(X) = X +(10, -10, -5, 0)$.
\end{enumerate}
\end{prop}
To construct $\Upsilon$, we need the following auxiliary diffeomorphims.
\begin{lem}\label{chi}
Given two points $Y, Z \in \mathbb{R}^4$ and three real numbers $a>b>c>0$ with $d(Y,Z) >2a$,
there exists a diffeomorphism $\chi[Y,Z, a,b,c] (X) = \chi(X)$ on $\mathbb{R}^4$ which has the following properties:
\begin{enumerate}
\def\labelenumi{($\chi$\theenumi)}
\item The support of  $\chi(X)$ is contained in $\mathcal{C}(Y, Z, a)$.
\item For $X \in \mathcal{B}(Y, c),  \Upsilon(X) = X +Z -Y$.
\end{enumerate}
\end{lem}

Let us see how one can construct $\Upsilon$ from $\chi$.
We define $l_n := 1.1-0.1n$,
\[
\chi_i(X) = \chi(C_i, C_{i+1}, l_{3i-2}, l_{3i-1}, l_{3i}) 
\]
for $i = 1, 2,3$ and put $\Upsilon = \chi_3 \circ \chi_2 \circ \chi_1$.
Then $\Upsilon$ satisfies ($\Upsilon$1) and ($\Upsilon$2).

\begin{proof}[Proof of lemma \ref{chi}]
By changing coordinate, we can assume $Y$ is the origin of $\mathbb{R}^4$ and $Z = (\zeta, 0, 0, 0)$
where $\zeta$ is some real number with $\zeta >2a$.
First we construct a diffeomorphism $\kappa(x)$ of $\mathbb{R}$ that satisfies the following conditions:
\begin{enumerate}
\def\labelenumi{($\kappa$\theenumi)}
\item The support of $\kappa$ is contained in $[-a, \zeta +a]$.
\item For $x \in [-c, c]$, $\kappa (x) = x +\zeta$.
\end{enumerate}
Then we define $\chi$ as follows.
We put $\rho_1(t) = \rho[-c, c, -b, b](t)$, $R(x,y,z,w) = \rho_1(y)\rho_1(z)\rho_1(w)$ and 
for $X = (x, y, z, w) \in \mathbb{R}^4$,
\[
\chi[Y,Z, a,b,c] (X) = \left( R(X)\kappa(x) + (1-R(X))x, y, z, w \right).
\]
It is not difficult to see $\chi$ satisfies the conditions required.

Let us see how to construct $\kappa(x)$.
We prepare a $C^{\infty}$-function $\eta (x)$ on $\mathbb{R}$ that satisfies the following properties:
\begin{enumerate}
\def\labelenumi{($\eta$\theenumi)}
\item $\eta(x) >0$ for all $x >0$.
\item $\eta(x) =1$ if $x < -b, -c < x < c$, or $x>\zeta+b$. 
\item $\int_{-b}^{-c}\eta(x) = \zeta+b-c$.
\item $\int_{c}^{\zeta+b}\eta(x) = b-c$.
\end{enumerate}  
Then $\kappa(x) := z + \int_{0}^{x}\eta(t)\,dt$ is a $C^{\infty}$-diffeomorphism satisfying ($\kappa$1) and ($\kappa$2).
Finally, let us see how to construct $\eta(t)$.
We fix a positive real number $e < (b-c)/2$. Let us define
\[
\rho_2(x)= \rho[-b+e, -c-e, -b, -c](x),  \rho_3(x)= \rho[c+e, \zeta +b-e, c, \zeta +b](x),
\]
and
\[
\eta(x; \alpha, \beta) = \exp(\alpha \rho_2(x) + \beta\rho_3(t)),
\]
where $\alpha, \beta$ are some real numbers. 
We show there exists $\alpha_1$, $\beta_1$ such that $\eta(x; \alpha_1, \beta_1)$
satisfies ($\eta$1)--($\eta$4).

($\eta$1) and ($\eta$2) holds for $\eta(x; \alpha, \beta)$ for all $\alpha$ and $\beta$. 
Thus let us consider ($\eta$3) and ($\eta$4).
For $x \in [-b, -c]$, $\eta(x; \alpha, \beta)$ is equal to  $\exp(\alpha \rho_2(x))$. 
Let us put $J(\alpha) := \int_{-b}^{-c}\eta(x; \alpha, \beta)$. Then one can check that
$J(\alpha) \to 0$ as $\alpha \to -\infty$, $J(\alpha) \to +\infty$ as $\alpha \to +\infty$
and $J(\alpha)$ is continuous and monotonously increasing.
So, by intermediate value theorem, there exists $\alpha_1$ such that $J(\alpha_1) = \zeta+b-c$.
In the similar way, one can find $\beta_1$ such that $\int_{c}^{\zeta+b}\eta(x; \alpha_1, \beta_1) = b-c.$
\end{proof}

\subsection{Proof of proposition \ref{local}}
Finally, we finish the proof of proposition \ref{local}.
\begin{proof}[Proof of proposition \ref{local}]
Let us put $c_{\Upsilon} := \min_{x \in \mathbb{R}^4 } m(\Lambda^3(d\Upsilon)(x))$.
Since $\Upsilon$ is a diffeomorphism and has compact support, $c_{\Upsilon} >0$.
We take $\lambda_0>0$ so that $ K:= c_{\Phi}c_{\Upsilon}\lambda_0  >1$ holds.
We put $\Omega := \Upsilon \circ \Phi_{\lambda_0}$
and show $\Omega$, $A$, $B$, $P$, $Q$ and $K$ satisfies ($w$1)--($w$7).
Indeed, 
\begin{itemize}
\item ($w$1) follows from ($\Phi$1) and ($\Upsilon$1).
\item ($w$2) follows from ($\Phi$2) and ($\Upsilon$1).
\item ($w$3) follows from ($\Phi$3), ($\Phi$4), ($\Phi$7) and ($\Upsilon$1).
\item ($w$4) follows from ($\Phi$5), ($\Phi$6), ($\Phi$8), ($\Phi$9), ($\Upsilon$1) and ($\Upsilon$2).
\item ($w$5) follows from ($\Phi$4) and ($\Upsilon$1).
\item ($w$6) follows from ($\Phi$7) and ($\Upsilon$1).
\item ($w$7) follows from ($\Phi$10) and the definition of $\lambda_0$.
\end{itemize}
So the proof is completed.
\end{proof}

\end{document}